%% file: lassas_oksanen-inverse_problem_for_the_wave_eq-2012_08m_10d.tex
\documentclass[12pt]{amsart}
\usepackage{main}
\usepackage{caption}
\usepackage{subcaption}

\def\grad{\text{grad}}
\def\bar{\overline}

\def\tilde{\widetilde}
\def\inter{\text{int}}

\def\Src{\mathcal S}
\def\Rec{\mathcal R}

\title[Inverse problem for wave equation]{Inverse problem for the Riemannian wave equation 
with Dirichlet data and Neumann data on disjoint sets}
\author{Matti Lassas and Lauri Oksanen}
\address{University of Helsinki, Department of Mathematics and Statistics, P.O. Box 68 FI-00014}
\date{\today}
\begin{document}
\begin{abstract}
We consider the inverse problem to determine 
a smooth compact Riemannian manifold with boundary $(M, g)$ 
from a restriction $\Lambda_{\Src, \Rec}$ of the Dirichlet-to-Neumann operator
for the wave equation on the manifold.
Here $\Src$ and $\Rec$ are open sets in $\p M$ and 
the restriction $\Lambda_{\Src, \Rec}$ corresponds to the case 
where the Dirichlet data is supported on $\R_+\times \Src$
and the Neumann data is measured on $\R_+\times \Rec$.
In the novel case where $\bar \Src \cap \bar \Rec = \emptyset$,
we show that $\Lambda_{\Src, \Rec}$ determines the manifold $(M,g)$ uniquely,
assuming that the wave equation is exactly controllable from the set of sources $\Src$. 
Moreover, we show that the exact controllability can be replaced by the Hassell-Tao condition
for eigenvalues and eigenfunctions, that is,
\begin{align*}
\lambda_j \le C \norm{\p_\nu \phi_j}_{L^2(\Src)}^2, \quad j =1, 2, \dots,
\end{align*}
where $\lambda_j$ are the Dirichlet eigenvalues and 
$(\phi_j)_{j=1}^\infty$ is an orthonormal basis of the corresponding eigenfunctions.
\end{abstract}
\maketitle

\section{Introduction}

Let $(M, g)$ be a smooth, connected and compact Riemannian manifold with boundary $\p M$.
We consider the wave equation with Dirichlet data $f \in C_0^\infty((0, \infty) \times \p M)$,
\begin{align}
\label{eq_wave}
&(\p_t^2 - \Delta_g )u(t,x) = 0, &\text{in $(0, \infty) \times M$},
\\\nonumber &u|_{(0, \infty) \times \p M} = f, &\text{in $(0, \infty) \times \p M$},
\\\nonumber &u|_{t = 0} = \p_t u|_{t = 0} = 0, &\text{in $M$},
\end{align}
and denote by $u^f = u(t,x)$ the solution of (\ref{eq_wave}).
For open and nonempty sets $\Src, \Rec \subset \p M$ and $T \in (0, \infty]$
we define the restricted Dirichlet-to-Neumann operator,
\begin{align*}
\Lambda_{M,g,\Src, \Rec}^T : f \mapsto \p_\nu u^f|_{(0,T) \times \Rec}, 
\quad f \in C_0^\infty((0, T) \times \Src).
\end{align*}
Often we write $\Lambda_{\Src, \Rec}^T=\Lambda_{M,g,\Src, \Rec}^T$
and $\Lambda_{\Src, \Rec} = \Lambda_{\Src, \Rec}^\infty$.
When $f$ is regarded as a boundary source, the operator $\Lambda_{\Src, \Rec}^T$ models 
boundary measurements for the wave equation with sources producing
the waves on $(0, T) \times \Src$
and the waves being observed on $(0, T) \times \Rec$.  
We consider the inverse boundary value problem to determine $(M,g)$ from $\Lambda_{M,g,\Src, \Rec}^T$.


A problem of this type is often called a complete boundary data problem 
if $\Src = \Rec = \p M$ and 
a partial boundary data problem if $\Src \ne \p M$ or $\Rec \ne \p M$. A sub-class of the partial boundary data problems
are the local data problems where  $\Src= \Rec \ne \p M$. 
The inverse problems with local data and 
the analogous partial boundary data problems with 
 $\Src \cap  \Rec \ne \emptyset$
have been studied broadly and
we will give below a brief review of this literature.
On the contrary, problems with $\overline \Src \cap \overline \Rec = \emptyset$,
that is, problems with disjoint partial data,
have remained open to large extent. 
We are aware of only two previous results:
in \cite{Rakesh2000} Rakesh proved that the coefficients of a wave equation 
on a one-dimensional interval are determined by 
boundary measurements with 
sources supported on one end of the interval and the waves observed on the other end,
and in \cite{Imanuvilov2011a} Imanuvilov, Uhlmann, and Yamamoto proved that 
the potential of a Schr\"odinger equation 
on a two-dimensional domain homeomorphic to a disc, where the boundary
is partitioned into eight clockwise-ordered parts $\Gamma_1,\Gamma_2,\dots,\Gamma_8$, is determined by boundary measurements with sources supported on $\Src=\Gamma_2\cup\Gamma_6$ and fields observed on $\Rec=\Gamma_4\cup\Gamma_8$.

Inverse problems with  partial boundary data  are encountered in mathematical
physics and in various applications.  For example in medical imaging
and in the geophysical imaging of the Earth, measurements can usually
be done only a part of the boundary.
Often it is not possible to observe fields on the same area where sources are controlled.
For example in oil exploration, explosives are used as sources
and hence it is difficult to measure waves near the sources.
Also, many inverse scattering problems, such as the transmission problems
on a line, are equivalent to disjoint partial data problems.


In this paper we consider the problem with $\overline \Src \cap \overline \Rec = \emptyset$
and show that the inverse problem to determine $(M, g)$ given
$\Lambda_{\Src, \Rec}$
has unique solution if the wave equation (\ref{eq_wave})
is exactly controllable from $\Src$.
We say that $(M,g)$ is exactly controllable from $\Src$ in time $T_0 > 0$ if the map
\begin{eqnarray}
\label{def_exact_controllability}
& &\mathcal U:
L^2((0,T_0) \times \Src) \to L^2(M) \times H^{-1}(M),
\\ & &\nonumber \mathcal U(f )=(u^f(T_0), \p_t u^f(T_0))
\end{eqnarray}
is surjective.
The condition by Bardos, Lebeau and Rauch gives a geometric characterization 
of exact controllability \cite{Bardos1992, Burq1997}.
In particular, if $M$ has a strictly convex boundary, 
then exact controllability is valid when every geodesic, 
continued by normal reflection on the boundary and having
length $T_0$, intersects $\Src$. 
We refer to \cite{Bardos1992} for the precise formulation of the geometric condition in the case that $\p M$ is non-convex.

For our purposes the exact controllability can be replaced by a spectral condition
that is strictly weaker in terms of the size of $\Src$.
Namely, let us denote the Dirichlet eigenvalues of $-\Delta_g$ by
\begin{align*}
0 < \lambda_1 < \lambda_2 \le \lambda_3 \dots \to \infty
\end{align*}
and the corresponding $L^2(M)$-normalized 
eigenfunctions by $\phi_j$. 
That is,
$(\phi_j,\phi_k)_{L^2(M)} = \delta_{jk}$ and 
\begin{align*}
- \Delta_g \phi_j = \lambda_j \phi_j\quad\hbox{on }M, 
\quad \phi_j|_{\p M} = 0.
\end{align*}
We say that the manifold $(M,g)$
satisfies the Hassell-Tao condition
for eigenvalues and eigenfunctions with the set $\Src\subset \p M$ if 
 there is $C_0>0$ such that for all orthonormal bases $(\phi_j)_{j=1}^\infty$ of eigenfunctions
\begin{align}
\label{def_spectral_cond}
\lambda_j \le C_0 \norm{\p_\nu \phi_j}_{L^2(\Src)}^2, \quad \hbox{for all }j =1, 2, \dots. 
\end{align}
%

\begin{figure}[t]
\begin{subfigure}[t]{3cm}
\centering
\def\svgwidth{3cm}
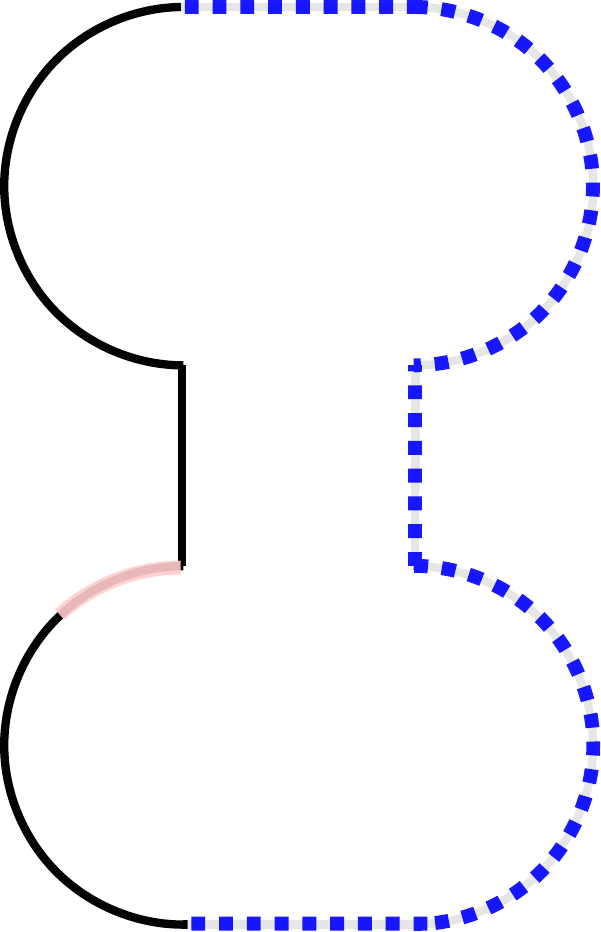
\end{subfigure}
\qquad
\begin{subfigure}[t]{3cm}
\centering
\def\svgwidth{2cm}
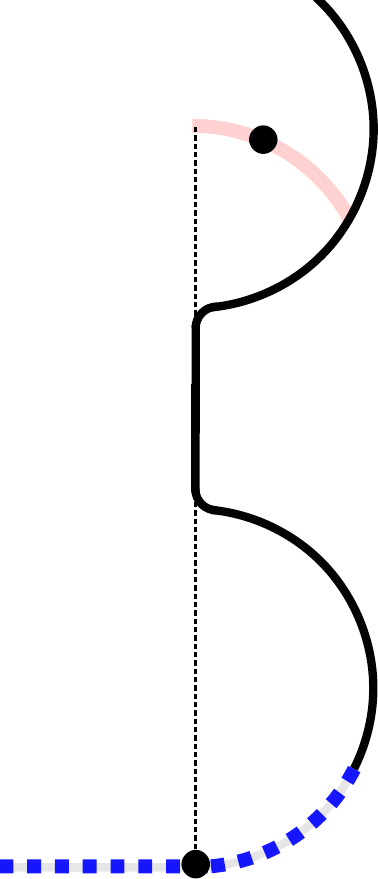
\end{subfigure}
\caption{
{\em On left}, the dogbone region with the Euclidiean metric satisfies the assumptions of Theorem \ref{thm_main} when $\Src$ is the dashed blue part of the boundary and 
$\Rec$ is the light red part of the boundary, see \cite[Fig. 6]{Bardos1992}.
{\em On right}, a detail of the dogbone region: 
for any point $y$ on the light red arc, 
the point $x$ is the closest point to $y$ on the dashed blue part of the boundary.
We overcome the difficulties arising from non-convexity
by using modified boundary distance functions in Section \ref{sec_cut_locus}.
}
\end{figure}

We denote by $d(x, y)$, $x,y \in M$, the Riemannian distance function of $(M, g)$.
Our main result is the following. 
\begin{theorem}
\label{thm_main}
Let $(M_1,g_1)$ and $(M_2,g_2)$ be $C^\infty$-smooth compact and connected
Riemannian manifolds with boundary 
and let $\Src_j\subset \p M_j$ and $\Rec_j\subset \p M_j$
be open non-empty sets with smooth boundaries for $j=1,2$. 
Suppose that there are diffeomorphisms $\Phi: \bar \Src_1 \to \bar \Src_2$ and
$\Psi:\bar \Rec_1 \to \bar \Rec_2$ and that
(H) or (H') holds, where
\begin{itemize}
\item[(H)] $(M_j,g_j)$, $j=1,2$, is exactly controllable from $\Src_j$ in time $T_0 > 0$ 
and there is $T > T_0 + 2 \max_{j=1,2} \max_{x \in M_j} d(x, \Rec_j)$ such that
\begin{align*}
\Lambda_{M_1,g_1,\Src_1, \Rec_1}^{T} f=\Psi^*(\Lambda_{M_2,g_2,\Src_2, \Rec_2}^{T}(\Phi_*f))
\end{align*}
for all $f\in C^\infty_0((0, T) \times \Src_1)$.
\item[(H')] $(M_j,g_j)$, $j=1,2$, satisfies the Hassell-Tao condition (\ref{def_spectral_cond}) with the set $\Src_j$ and
\begin{align*}
\Lambda_{M_1,g_1,\Src_1, \Rec_1}^\infty f=\Psi^*(\Lambda_{M_2,g_2,\Src_2, \Rec_2}^\infty(\Phi_*f))
\end{align*}
for all $f\in C^\infty_0((0, \infty) \times \Src_1)$.
\end{itemize}
Then $(M_1,g_1)$ and $(M_2,g_2)$ are isometric and there is such an isometry $F:M_1\to M_2$
that $F|_{\Src_1}=\Phi$ and $F|_{\Rec_1}=\Psi$. 
\end{theorem}

Notice that $M^\inter$ is not assumed to be known a priori and
a part of the proof is to construct $M$ as a smooth manifold.
The same is true for $\p M \setminus \bar{\Src \cup \Rec}$.

While proving the above theorem, we develop a new geometric technique that we 
call modified boundary distance functions. This technique allows us to overcome difficulties arising from possible non-trans\-versality between the geodesic flow and the boundary, see Figure 1. Such difficulties are present in various inverse problems, see e.g. \cite{Isozaki2010, Stefanov2008}.

\subsection{Previous results}

The first results for inverse problems for the wave equation and  the equivalent inverse
problems 
 the heat or the Schr\"odinger equations go back to  the end of 50's
 when  Krein studied  the one-dimensional inverse problem for an inhomogeneous string,
$u_{tt}-c^2(x)u_{xx}=0$, see e.g. \cite{Kreuin1951}. 
In his works, causality was transformed into analyticity 
of the Fourier transform of the solution. Later, 
in the late 60's,  Blagovestchenskii \cite{Blagovevsvcenskiui1969, Blagovevsvcenskiui1971}
developed a method  to solve one dimensional inverse problems that was
based on finite speed of wave propagation.
In the late 80's  Belishev \cite{Belishev1987} developed the boundary control method 
by combining the finite speed arguments with control theory to
solve  inverse problems for the  wave equation in domains
of the Euclidean space $\R^n$.  
A turning point in the study of inverse problems 
for wave equation happened in 1995 when  
 Tataru  proved 
a Holmgren-type uniqueness theorem for wave equations with non-analytic coefficients
\cite{Tataru1995, Tataru1999}.
This and the earlier results on the boundary control method by Belishev and 
Kurylev \cite{Belishev1992}  solved the inverse problem for the wave equation on
a Riemannian manifold with complete data.
 For later development of 
the geometric boundary control method, see \cite{Anderson2004,Belishev2007, Katchalov1998, Katchalov2001, Kurylev1997}.
In all the above results measurements were assumed to be
given either on the whole boundary or it was assumed that waves are observed on the same
sets where sources are supported, that is, $\Src=\Rec$ in our notation.
 For the wave equation, the problem where the closures of the source 
 domain $\Src$ and  the observation
 domain $\Rec$ do not intersect have been studied only in few papers, see \cite{Rakesh2011} and the references therein, typically in the one-dimensional or
radially symmetric cases. 
 In \cite{Lassas2010} we studied the case when there are three disjoint
 domains $\Sigma_1,\Sigma_2,\Sigma_3\subset \p M$ and assumed
 that
 all three Dirichlet-to-Neumann maps $\Lambda_{\Sigma_1,\Sigma_2}$,
 $\Lambda_{\Sigma_2,\Sigma_3}$, and $\Lambda_{\Sigma_3,\Sigma_1}$
 are known.

The steady state solutions of the wave equation satisfy an elliptic
equations and thus the inverse problems for elliptic equations
can be in many cases considered as special cases of hyperbolic
inverse problems with restricted data. A paradigm problem of this type is  Calder\'on's inverse problem \cite{Calder'on1980} that concerns 
the determination of the conductivity from the elliptic Dirichlet-to-Neumann map.
A smooth isotropic conductivity in a bounded
domain of $\R^n$, $n\geq 3$, is determined  by the elliptic Dirichlet-to-Neumann map
as was shown in the seminal paper of Sylvester and Uhlmann \cite{Sylvester1987}.
 In two dimensions the first unique identifiability
result was proven in \cite{Nachman1996} for $C^2$ conductivities and in \cite{Astala2006}
the problem was solved for $L^\infty$ conductivities.
The corresponding inverse problems
for the elliptic Schr\"odinger equation has
been solved in \cite{Bukhgeim2008, Sylvester1987}.
%
%
For anisotropic smooth conductivity (corresponding to a general Riemannian metric)
 in $\Omega\subset \R^n$, 
Calder\'on's inverse problem was solved in two
dimensions in \cite{Sylvester1990} using
the isotropic result \cite{Nachman1996}. The needed regularity
was reduced to $L^\infty$ later in \cite{Astala2005}.
In the case of dimension $n\ge 3$,
Calder\'on's inverse problem is of geometrical nature and makes sense for general
compact Riemannian manifolds with boundary, as was pointed out in \cite{Lee1989}.
This problem remains open, however, and we refer to \cite{DosSantosFerreira2009, Henkin2008, Lassas2003a} for partial results.


The partial data problem for the isotropic elliptic equation 
in $\Omega\subset \R^n$, $n \ge 3$, 
 under certain geometrical
restrictions on the sets $\Src$ and $\Rec$, 
 has been
solved by Kenig, Sj\"ostrand and Uhlmann \cite{Kenig2007}. 
Also,  the local data problem for anisotropic elliptic equations
 in $\Omega\subset \R^2$ was recently solved
 by Imanuvilov, Uhlmann and Yamamoto 
\cite{Imanuvilov2010a, Imanuvilov2011}.
For two-dimensional manifolds it was shown in \cite{Lassas2001} that
 the local boundary data  for the Laplace-Beltrami equation with $\Src=\Rec\ne \p M$ 
 determines uniquely
the manifold and the 
conformal class of the metric. Later the amount of needed measurements
have been reduced in 
\cite{Henkin2007, Henkin2008, Henkin2010, Henkin2011}.
For Riemannian surfaces the determination 
of the potential in a Riemann-Schr\"odinger equation
 with local data 
corresponding to $\Src=\Rec\ne \p M$  has been
solved in \cite{Guillarmou2011}.
The above mentioned \cite{Imanuvilov2011a}
is the only result with $\Src \cap \Rec = \emptyset$ concerning elliptic equations
that we are aware of.

\section{Outline of the arguments and notations}

We begin by showing in Section \ref{sec_controllability} that the condition (H) in Theorem \ref{thm_main}
implies the condition (H'). That is, we show that
the Hassell-Tao condition (\ref{def_spectral_cond})
is implied by the exact controllability (\ref{def_exact_controllability}),
and that the operator $\Lambda_{\Src, \Rec}^\infty$
is determined by the operator $\Lambda_{\Src, \Rec}^T$ when 
\begin{align}
\label{hypothesis_T}
T > T_0 + 2 \max_{x \in M} d(x, \Rec)
\end{align}
and (\ref{eq_wave}) is exactly controllable from $\Src$ in time $T_0$.
If we take $\Src = \p M$ and assume that $M$ is non-trapping, then (\ref{def_spectral_cond}) was proved 
by Hassell and Tao \cite{Hassell2002} without using exact controllability.
In the erratum for \cite{Hassell2002}, 
it was noted that exact controllability yields (\ref{def_spectral_cond})
and this observation was attributed to Nicolas Burq.
We give a short proof for this fact in Section \ref{sec_controllability} and show that (\ref{def_spectral_cond})
is strictly weaker than exact controllability in terms of the size of $\Src$. 

We will prove Theorem \ref{thm_main} in five steps that we will describe next.
We denote $SM := \{\xi \in TM;\ |\xi|_g = 1\}$, that is, $SM$ is the unit sphere bundle, and define
$\p_\pm S M := \{\xi \in \p SM;\ (\xi, \mp \nu)_g > 0 \}$,
where $\nu$ is the interior unit normal vector on $\p M$.
We define the {\em exit time} for $(x, \xi) \in SM \setminus \bar{\p_+ S M}$,
\begin{align*}
\tau_M(x, \xi) := \inf \{ s \in (0, \infty);\ \gamma(s; x, \xi) \in \p M \},
\end{align*}
where $\gamma(\cdot; x, \xi)$ is the geodesic with the initial data $\gamma(0) = x$, $\dot \gamma(0) = \xi$.
Moreover, we define the strip
\def\NN{\mathcal N}
\begin{align}\label{Matti's new formula 1}
\NN_\Rec &:= \{(s, y) \in (0, \infty) \times \Rec;\ s < \sigma_\Rec(y) \},
\\ \nonumber
\sigma_\Rec(y) &:= \max \{ s \in (0, \tau_M(y, \nu)];\ d(\gamma(s; y, \nu), \Rec) = s\},
\end{align}
and denote by $M_\Rec$ the image of $\NN_\Rec$ under the map 
\begin{align}
\label{boundary_normal_coords_s_y}
(s, y) \mapsto \gamma(s; y, \nu).
\end{align}
The first step is to show that $\Lambda_{\Src, \Rec}$
determines $(M_\Rec, g)$. 
That is, we reconstruct a piece of $(M,g)$ in the boundary normal coordinates. 
Notice that $\sigma_\Rec(y) > 0$, see e.g. \cite[p. 50]{Katchalov2001}. 

The second step is to show that $\Lambda_{\Src, \Rec}$ and $(M_\Rec, g)$
determine the map 
\begin{equation*}
\Lambda_{\Src, B} : f \mapsto u^f|_{(0,\infty) \times B}, 
\quad f \in C_0^\infty((0, \infty) \times \Src),
\end{equation*}
where $B \subset M_\Rec^\inter$ is a small ball sufficiently far away from 
the topological boundary of $M_\Rec$.
The first two steps can be thought as bootstrapping steps that give us the interior data $\Lambda_{\Src, B}$. The remaining three steps employ only interior data,
and we will reconstruct the unknown manifold $(M,g)$ 
by extending the known subset iteratively. 

We denote by $x$ the center of $B$ and define
\begin{align*}
\NN_B &:= (0, \sigma^B) \times S_x M,
\\
\sigma^B &:= \min_{\xi \in S_x M} 
\max \{ s \in (0, \tau_M(x, \xi)];\ d(\gamma(s; x, \xi), x) = s\}.
\end{align*}
We denote by $M_B$ the image of $\NN_B$ under the map 
\begin{align}
\label{normal_coords_s_xi}
(s, \xi) \mapsto \gamma(s; x, \xi).
\end{align}
The third step is to show that $\Lambda_{\Src, B}$
determines $(M_B, g)$. 
That is, we reconstruct a piece of $(M,g)$ in the geodesic normal coordinates. 
Essentially the same geometric method can be used to reconstruct
$(M_\Rec, g)$ and $(M_B, g)$ from $\Lambda_{\Src, \Rec}$
and $\Lambda_{\Src, B}$, respectively.
We will describe this method in Section \ref{sec_local_construction}.

The fourth step is very similar with the second step. 
We show that $\Lambda_{\Src, B}$ and $(M_B, g)$ determine
$\Lambda_{\Src, B'}$ for a sufficiently small ball $B' \subset M_B$.
The fifth step is to iterate the third and the fourth step and 
to glue the local reconstructions together. 
We will give the details of the second, fourth and fifth steps in Section \ref{sec_global_reconstruction}.

\section{Exact controllability and testing weak convergence
of sequences of waves}
\label{sec_controllability}

Let us recall that for open $\Gamma \subset \p M$ and $f \in C_0^\infty((0, \infty) \times \Gamma)$
we denote by $u^f = u$ the solution of 
\begin{align*}
&\p_t^2 u - \Delta_g u = 0, &\text{in $(0, \infty) \times M$},
\\\nonumber &u|_{(0, \infty) \times \p M} = f, &\text{in $(0, \infty) \times \p M$},
\\\nonumber &u|_{t = 0} = \p_t u|_{t = 0} = 0, &\text{in $M$}.
\end{align*}
We extend the notation $u^f$ for open $\Gamma \subset M^\inter$ and $f \in C_0^\infty((0, \infty) \times \Gamma)$
as the solution $u^f = u$ of 
\begin{align*}
&\p_t^2 u - \Delta_g u = f, &\text{in $(0, \infty) \times M$},
\\\nonumber &u|_{(0, \infty) \times \p M} = 0, &\text{in $(0, \infty) \times \p M$},
\\\nonumber &u|_{t = 0} = \p_t u|_{t = 0} = 0, &\text{in $M$}.
\end{align*}

\begin{lemma}[Blagove{\v{s}}{\v{c}}enski{\u\i}'s identity]
\label{lem_blago}
Let $T > 0$ and 
let $\Gamma$ be open either in $\p M$ or in $M^\inter$.
Then for 
\begin{align*}
\psi \in C_0^\infty((0, \infty) \times \Src), \quad
f \in C_0^\infty((0, \infty) \times \Gamma).
\end{align*}
we have
\begin{align}
\label{blago_inner_prod}
(u^f(T), u^\psi(T))_{L^2(M)} 
= (f, (J \Lambda_{\Src, \Gamma} - R \Lambda_{\Src, \Gamma} R J) \psi)_{L^2((0, T) \times \Gamma)},
\end{align}
where $R\psi(t) := \psi(T - t)$ and
$J\psi(t) := \frac{1}{2} \int_t^{2T - t} \psi(s) ds$.
\end{lemma}

For a proof in the case $\Gamma \subset \p M$ see e.g. \cite[Lem. 4.15]{Katchalov2001}.
The case $\Gamma \subset M^\inter$ is analogous. 
In the lemma, the Riemannian volume measures are used on $(M, g)$ and on $(\p M, g)$.
As we do not assume $g|_{\p M}$ to be known a priori, the right-hand side of (\ref{blago_inner_prod}) for $\Gamma = \Rec$
is not trivially determined by the boundary measurement data $\Lambda_{\Src, \Rec}$.
However, for our purposes it is sufficient to choose an arbitratry positive smooth measure $d\tilde S$ on $\Rec$. 
Then there is strictly positive $\mu \in C^\infty(\Rec)$ such that 
\begin{align}
\label{def_tilde_dS}
d \tilde S = \mu d S,
\end{align}
where $d S$ is the Riemannian volume measure of $(\Rec, g)$. Moreover,
\begin{align*}
(f, K \psi)_{L^2((0, T) \times \Rec; dt \otimes d\tilde S)}
&=
(\mu f, K \psi)_{L^2((0, T) \times \Rec)}
\\\nonumber&=
(u^{\mu f}(T), u^\psi(T))_{L^2(M)},
\end{align*}
where $K := J \Lambda_{\Src, \Rec} - R \Lambda_{\Src, \Rec} R J$.

Let us assume for a moment that the exact controllability (\ref{def_exact_controllability}) holds
and let $(f_j)_{j=1}^\infty \subset C_0^\infty((0, \infty) \times \Rec)$.
Then the functions $u^{f_j}(T_0)$ tend weakly to zero in $L^2(M)$ as $j \to \infty$ if and only if 
the inner products (\ref{blago_inner_prod}) with $f=f_j$ and $T=T_0$ tend to zero 
for all $\psi \in L^2((0, T_0) \times \Src)$.
In this section we will describe a method to determine 
if $(u^{\mu f_j}(T))_{j=1}^\infty$ 
is weakly convergent under the weaker assumption (\ref{def_spectral_cond}).
Let us point out that the assumption (\ref{def_spectral_cond}) is needed only in this step. 

If $\Rec = \Src$ then Lemma \ref{lem_blago} allows us to compute 
$\norm{u^{\mu f_j}(T)}_{L^2(M)}$ and
we can easily determine if $(u^{\mu f_j}(T))_{j=1}^\infty$ 
tends to zero. In this case the assumption (\ref{def_spectral_cond}) is not needed
since we can replace the exact controllability by the approximate controllability,
the latter of which holds for arbitrary $(M, g)$ by Tataru's unique continuation \cite{Tataru1995}, see e.g. \cite[Th. 3.10]{Katchalov2001}. 
The case $\Rec = \Src$ was solved originally in \cite{Katchalov1998}.

\subsection{Exact controllability and the Hassell-Tao condition}

We recall that the eigenvalues of the positive Laplace-Beltrami operator 
$-\Delta_g$ with the domain $H^2(M) \cap H_0^1(M)$
are denoted by $0 < \lambda_1 < \lambda_2 \le \lambda_3 \dots \to \infty$
and the corresponding $L^2(M)$-normalized eigenfunctions by $\phi_j$. 

It is well-known, see e.g. \cite{Bardos1992}, 
that the exact controllability (\ref{def_exact_controllability})
is equivalent with the continuous observability inequality,
\begin{align}
\label{continuous_observability}
\norm{(w_0, w_1)}_{H_0^1(M) \times L^2(M)} \le C \norm{\p_\nu w}_{L^2((0,T_0) \times \Src)}, 
\end{align}
where $w$ is the solution of the wave equation
\begin{align}
\label{eq_wave_w}
&\p_t^2 w - \Delta_g w = 0, &\text{in $(0, T_0) \times M$},
\\\nonumber &w|_{(0, T) \times \p M} = 0, &\text{in $(0, T_0) \times \p M$},
\\\nonumber &w|_{t = 0} = w_0, \quad \p_t w|_{t = 0} = w_1 &\text{in $M$}.
\end{align}

\begin{lemma}
Suppose that (\ref{continuous_observability}) holds.
Then there is $C > 0$ such that
\begin{align}
\label{spec_lower_bound}
\sqrt{\lambda_j + 1} \le C \norm{\phi_j}_{L^2(\Src)},
\quad \text{for all $j =1, 2, \dots$.}
\end{align}
\end{lemma}
\begin{proof}
Notice that if (\ref{continuous_observability}) holds
then it holds also when $T_0$ is replaced by a larger time. 
Let $w_0 = \phi_j$ and $w_1 = 0$ in (\ref{eq_wave_w}).
By writing the corresponding solution $w$ in 
the eigenbasis $(\phi_j)_{j=1}^\infty$ we see that
\begin{align}
\label{sol_in_eigenbasis}
w(t, x) =& \cos(\sqrt{\lambda_j} t) \phi_j(x).
\end{align}
Moreover, 
\begin{align*}
\norm{(w_0, w_1)}_{H_0^1(M) \times L^2(M)}^2
&= \norm{\phi_j}_{H_0^1(M)}^2
= (d \phi_j, d \phi_j)_{L^2(M)} + 1 
\\&= (\Delta_g \phi_j, \phi_j)_{L^2(M)} + 1
= \lambda_j + 1.
\end{align*}
If $T > \frac{1}{2 \sqrt{\lambda_1}}$ then
\begin{align*}
\int_0^T \cos^2(\sqrt{\lambda_j} t) dt
=
\frac{T}{2} + \frac{\sin(2 \sqrt{\lambda_j} T)}{4 \sqrt{\lambda_j}}
\ge \frac{T}{2} - \frac{1}{4 \sqrt{\lambda_1}} > 0.
\end{align*}
If also $T \ge T_0$, then
we have by (\ref{sol_in_eigenbasis}) and (\ref{continuous_observability}) that
\begin{align*}
\lambda_j + 1 
&= 
\norm{(w_0, w_1)}_{H_0^1(M) \times L^2(M)}^2 
\le 
C \norm{\p_\nu w}_{L^2((0,T) \times \Gamma)}^2
\\&\le
C' \frac{\norm{\p_\nu w}_{L^2((0,T) \times \Gamma)}^2}
{\int_0^T \cos^2(\sqrt{\lambda_j} t) dt}
= 
C' \int_\Gamma \ll(\p_\nu \phi_j \rr)^2 dS.
\end{align*}
\end{proof}

\begin{example}
Let $(M, g)$ be the Euclidean unit disc in $\R^2$ and denote
\begin{align*}
\Gamma_s := \{ e^{i 2 \pi \theta};\ \theta \in (0, s)\}.
\end{align*}
Exact controllability from $\Gamma_s$
holds for $s > 1/2$ and
does not hold for $s < 1/2$, see e.g. \cite{Bardos1992}.
However, the condition (\ref{spec_lower_bound}) holds on $\Src = \Gamma_s$ for $s \ge 1/4$.
\end{example}
\begin{proof}
An orthonormal basis of eigenfunctions can be chosen in the polar coordinates $(r, \theta) \in (0, 1] \times (0, 2\pi]$ as
\begin{align*}
\phi_{mn1} = c_{nm} \cos(m \theta) J_m(z_{mn} r), 
\quad 
\phi_{mn2} = c_{nm} \sin(m \theta) J_m(z_{mn} r),
\quad 
\end{align*}
where $c_{nm}$ is a normalization constant and 
$z_{mn}$ is the $n$th positive zero of the $m$th Bessel function $J_m$.
The corresponding eigenvalues are $\lambda_{nmj} = z_{mn}^2$
and we have used the indices $m = 0, 1, 2, \dots$, $n = 1, 2, \dots$ and $j=1,2$.
Notice that $z_{mn} \ne z_{m'n'}$ if $(m,n) \ne (m',n')$, see e.g. \cite[p. 484]{Watson1944}.
Thus any normalized eigenfunction $\phi$ corresponding to the eigenvalue $z_{mn}^2$
can be written as $\phi = a \phi_{mn1} + b \phi_{mn2}$
with $a^2 + b^2 = 1$. Moreover, for $m > 0$,
\begin{align*}
\norm{\p_\nu \phi}_{L^2(\Gamma_s)}^2 
&= c_{nm}^2 z_{mn}^2 (J_m'(z_{mn}))^2 \int_0^{2\pi s} (a\cos(m \theta)+ b\sin(m \theta))^2 d\theta.
\end{align*}
In particular,
\begin{align*}
&\int_0^{\pi/2} (a\cos(m \theta)+ b\sin(m \theta))^2 d\theta
=
\frac{-2 (-1+(-1)^m) a b+ m \pi}{4 m}
\\&\quad\quad\ge
\frac{m \pi - 2}{4 m}
>
\frac{\pi}{12}
= \frac{1}{12} \int_0^{2\pi} (a\cos(m \theta)+ b\sin(m \theta))^2 d\theta.
\end{align*}
Thus the lower bound (\ref{spec_lower_bound}) holds for $\Gamma_{1/4}$
with the constant $C /12$ where $C$ is the corresponding constant for $\Gamma_{1}$.
\end{proof}

\subsection{Testing weak convergence of sequences of waves}
\label{sec_weak_convergence}

\begin{lemma}
\label{lem_weak_conv}
Let $\Gamma \subset \p M$ be open and nonempty and 
let 
\begin{align*}
T > \max_{x \in M} d(x, \Gamma).
\end{align*}
A sequence $(v_l)_{l = 1}^\infty \subset L^2(M)$ converges to zero weakly in $L^2(M)$
if and only if both (a) and (b) hold, where
\begin{itemize}
\item[(a)] For all sequences $(\psi_m)_{m=1}^\infty \subset C_0^\infty((0, T) \times \Gamma)$ such that
the sequence $(u^{\psi_m}(T))_{m=1}^\infty \subset L^2(M)$ is bounded,
there is $C > 0$ satisfying
\begin{equation*}
|(v_l, u^{\psi_m}(T))_{L^2(M)}| \le C
\quad \text{for all $l, m = 1, 2, \dots$}.
\end{equation*}
\end{itemize}
\begin{itemize}
\item[(b)] $\lim_{l \to \infty} (v_l, u^\psi(T))_{L^2(M)} = 0$ for all $\psi \in C_0^\infty((0, T) \times \Gamma)$.
\end{itemize}
\end{lemma}
\begin{proof}
In the proof we omit writing $L^2(M)$ as a subscript.
If $(v_l)_{l=1}^\infty$ is weakly convergent to zero then 
(b) holds trivially and 
(a) holds since weakly convergent sequences are bounded.

Let us assume (a) and (b).
We will first show that $(v_l)_{l=1}^\infty$ is bounded in $L^2(M)$.
Let $v_0 \in L^2(M)$. 
Tataru's unique continuation \cite{Tataru1995}
implies approximate controllability, see e.g. \cite[Th. 3.10]{Katchalov2001}. 
Thus there is $(\psi_m)_{m=1}^\infty \subset C_0^\infty((0, T) \times \Gamma)$
such that $u^{\psi_m}(T) \to v_0$ in $L^2(M)$.
Hence for all $l$
\begin{equation*}
|(v_l, v_0)| = \lim_{m \to \infty} |(v_l, u^{\psi_m}(T))| \le C,
\end{equation*}
where $C > 0$ is independent of $l$.
Thus $(v_l)_{l=1}^\infty$ is weakly bounded, and hence 
bounded in the norm.

Let $\epsilon > 0$ and 
fix $m$ such that $\norm{v_0 - u^{\psi_m}(T)} \le \epsilon$.
Then for large $l$
\begin{align*}
|(v_l, v_0)| 
&\le \sup_l \norm{v_l} \norm{v_0 - u^{\psi_m}(T)} 
+ (v_l, u^{\psi_m}(T))
\\&\le \sup_l \norm{v_l} \epsilon + \epsilon.
\end{align*}
As $\epsilon > 0$ and $v_0 \in L^2(M)$ are arbitrary, 
$(v_l)_{l=1}^\infty$ converges weakly to zero.
\end{proof}

Let us reindex the orthonormal basis $(\phi_j)_{j=1}^\infty$ of Dirichlet eigenfunctions so that
for $j=1,2,\dots,$ the functions 
\begin{equation*}
\phi_{jk}, \quad k=1,2,\dots,K_j
\end{equation*}
span the space of eigenfunctions corresponding to the eigenvalue $\lambda_j$.
Here $0 < \lambda_1 < \lambda_2 < \lambda_3 \dots \to \infty$
and $K_j$ is the multiplicity of $\lambda_j$.
Let us choose a positive smooth measure $d \tilde S$ on $\bar \Src$. 
Then there is a strictly positive function $\mu \in C^\infty(\bar \Src)$ such that (\ref{def_tilde_dS}) holds.
As explained in \cite[pp. 5-6]{Lassas2010} 
the Fourier transform of 
the operator $\Lambda_{\Src, \Rec}$ 
with respect to the time variable is a meromorphic map,
and its poles and residues determine the Dirichlet eigenvalues 
$\lambda_j$
and also the spaces
\begin{equation*}
E_j := \linspan \{\mu^{-1} \p_\nu \phi_{jk}|_{\Src};\ k=1,2, \dots, K_j \} 
\subset C^\infty(\bar \Src),
\end{equation*}
for each $j=1, 2, \dots$.

\begin{proposition}
\label{prop_weak_boundedness}
Let $T > 0$ and $(\psi_m)_{m = 1}^\infty \subset C_0^\infty((0, \infty) \times \Src)$ 
and suppose that there is $C_0 > 0$ such that
\begin{equation*}
\sqrt{\lambda_j} \le C_0 \norm{\p_\nu \phi_{jk}}_{L^2(\Src)},
\quad k = 1, \dots, K_j,\ j=1, 2, \dots.
\end{equation*}
Then the following are equivalent
\begin{itemize}
\item[(i)] The sequence $u^{\psi_m}(T)$, $m=1, 2, \dots$, is bounded in $L^2(M)$.
\item[(ii)] For all $C_1 >0 $ and
$(e_j)_{j=1}^\infty \subset C^\infty(\Src)$ satisfying  
\begin{equation}
\label{pseudo_basis}
e_j \in E_j, 
\quad
\norm{e_j}_{L^2(\Src; d\tilde S)} \le C_1 \sqrt{\lambda_j},
\end{equation}
there is $C_2 > 0$ such that 
\begin{equation*}
\sup_{m} \sum_{j=1}^\infty 
\ll( \int_0^T s_j(t) \int_\Src \psi_m(t, x) e_j(x) d \tilde S(x) dt \rr)^2 \le C_2,
\end{equation*}
where 
$s_j(t) := \sin(\sqrt{\lambda_j}(T - t)) /\sqrt{\lambda_j}$.
\end{itemize}
\end{proposition}
\begin{proof}
Let (i) hold and let $e_j$, $j=1, 2, \dots$, satisfy (\ref{pseudo_basis}).
As the choice of the orthonormal basis $\{\phi_{jk}; k =1,2, \dots, K_j\}$
in the $j$th eigenspace is unique only up to a rotation, we may assume after a rotation that 
\begin{equation*}
\mu e_j = c_j \p_\nu \phi_{j1}|_\Src, 
\quad j=1, 2, \dots, 
\end{equation*}
where $c_j > 0$ are some constants.
Moreover, the sequence $(c_j)_{j=1}^\infty$ is bounded. 
Indeed, 
\begin{align*}
c_j \frac{\sqrt{\lambda_j}}{C_0} \le \norm{c_j \p_\nu \phi_{j1}}_{L^2(\Src)}
= \norm{\mu e_j}_{L^2(\Src)}
\le C \norm{e_j}_{L^2(\Src; d\tilde S)}
\le C C_1 \sqrt{\lambda_j}.
\end{align*}
Denoting $C' := (C C_0 C_1)^2$ we have
\begin{align*}
&\sum_{j=1}^\infty \ll( \int_0^T s_j \int_\Src \psi_m e_j d\tilde S dt \rr)^2
=
\sum_{j=1}^\infty c_j^2 
\ll(\int_0^T s_j \int_\Src \psi_m \p_\nu \phi_{j1} dS dt \rr)^2
\\&\quad\le C'
\sum_{j=1}^\infty (u^{\psi_m}(T), \phi_{j1})_{L^2(M)}^2
\le C' \norm{u^{\psi_m}(T)}_{L^2(M)}^2,
\end{align*}
and (ii) holds.

Let us now assume (ii) and let $v \in L^2(M)$.
We denote by $P_j$ the orthogonal projection onto the $j$th eigenspace.
We may rotate again the basis 
$\{\phi_{jk}; k =1,2, \dots, K_j\}$
so that 
\begin{align*}
\phi_{j1} = \frac{P_j v}{\norm{P_j v}_{L^2(M)}}, \quad \text{for all $j$ satisfying $P_j v \ne 0$}.
\end{align*}
Then $(v, \phi_{jk}) = 0$ for all $k \ge 2$ and all $j$.

We may choose $e_j := \mu^{-1} \p_\nu \phi_{j1}|_{\Src}$ in (ii).
Indeed, 
\begin{align*}
\norm{\mu^{-1} \p_\nu \phi_{j1}}_{L^2(\Src; d\tilde S)} 
\le C \norm{\p_\nu \phi_{j1}}_{L^2(\p M)}
\le C_1 \sqrt{\lambda_j},
\end{align*}
where the second inequality holds by \cite{Hassell2002}.
We have
\begin{align*}
|(v, u^{\psi_m}(T))|^2
&=
|\sum_{j=1}^\infty (v, \phi_{j1}) 
\int_0^T s_j \int_\Src \psi_m \p_\nu \phi_{j1} dS dt|^2
\\&\le 
\sum_{j=1}^\infty (v, \phi_{j1})^2
\sum_{j=1}^\infty \ll( \int_0^T s_j \int_\Src \psi_m e_j \mu dS dt \rr)^2
\\&\le \norm{v}_{L^2(M)}^2 
\sup_m \sum_{j=1}^\infty \ll( \int_0^T s_j \int_\Src \psi_m e_j d\tilde S dt \rr)^2
\end{align*}
and the sequence $u^{\psi_m}(T)$, $m = 1, 2, \dots$, is weakly bounded.
\end{proof}

\subsection{Continuation of the data in time}

\begin{lemma}
Let $\Src, \Rec \subset \p M$ be open and non-empty
and suppose that (\ref{eq_wave}) is exactly controllable from $\Src$ in time $T_0$.
If $T$ satisfies (\ref{hypothesis_T})
then $\Lambda_{\Src, \Rec}^T$ determines $\Lambda_{\Src, \Rec}^\infty$.
\end{lemma}
\begin{proof}
Let $\delta \in (0, T_0)$ satisfy 
\begin{align*}
T > T_0 + 2 \max_{x \in M} d(x, \Rec) + \delta,
\end{align*}
and let $f \in C_0^\infty((0, T+\delta) \times \Src)$.
We choose $h \in C_0^\infty((0, 2\delta) \times \Src)$
and $h' \in C_0^\infty((\delta, T+\delta) \times \Src)$
such that $f = h + h'$.
As the coefficients of (\ref{eq_wave}) are time independent,
we may translate in time and see that 
$\Lambda_{\Src, \Rec} h'(t)$, $t \in (\delta, T+\delta)$, is determined 
by $\Lambda_{\Src, \Rec}^T$.

We recall that $\Lambda_{\Src, \Rec}^T : L^2((0, T) \times \Src) \to H^{-1}((0, T) \times \Rec)$
is continuous \cite{Lasiecka1986}.
Let us suppose that $\tilde h \in L^2((\delta, T_0 + \delta) \times \Src)$
satisfies 
\begin{align}
\label{def_h_tilde}
\Lambda_{\Src, \Rec}^T h(t) = \Lambda_{\Src, \Rec}^T \tilde h(t),
\quad \text{for $t \in (T_0 + \delta, T)$}.
\end{align}
Notice that such a function $\tilde h$ exists. Indeed, 
by exact controllability (\ref{def_exact_controllability}) there is 
$\tilde h \in L^2((\delta, T_0 + \delta) \times \Src)$
such that 
\begin{align*}
(u^h(T_0 + \delta), \p_t u^h(T_0 + \delta)) = 
(u^{\tilde h}(T_0 + \delta), \p_t u^{\tilde h}(T_0 + \delta)).
\end{align*}
As also $h(t) = \tilde h(t) = 0$ for $t > T_0 + \delta$,
we have $u^h(t) = u^{\tilde h}(t)$
for $t > T_0 + \delta$.
We have shown that there is $\tilde h \in L^2((\delta, T_0 + \delta) \times \Src)$
satisfying (\ref{def_h_tilde}).

By (\ref{def_h_tilde}) the Cauchy data of $u^h$ and $u^{\tilde h}$ coincide on $(T_0 + \delta, T) \times \Rec$. 
Thus Tataru's unique continuation 
\cite{Tataru1995} implies that $u^h(t + T_0 + \delta) = u^{\tilde h}(t + T_0 + \delta)$
for $t$ near $(T - T_0 - \delta) / 2$,  see e.g. \cite[Th. 3.10]{Katchalov2001}. 
In particular,
$u^h(t) = u^{\tilde h}(t)$
for $t > T - \max_{x \in M} d(x, \Rec)$.
Analogously to the above case of $\Lambda_{\Src, \Rec}h'(t)$,
we see that $\Lambda_{\Src, \Rec}h(t) = \Lambda_{\Src, \Rec}\tilde h(t)$, $t \in (T, T+\delta)$,
is determined by $\Lambda_{\Src, \Rec}^T$.
In particular, we have determined 
\begin{align*}
\Lambda_{\Src, \Rec} f(t) = \Lambda_{\Src, \Rec}h(t) + \Lambda_{\Src, \Rec}h'(t),
\quad t \in (T, T+\delta).
\end{align*}
That is, we have shown that $\Lambda_{\Src, \Rec}^T$ determines
$\Lambda_{\Src, \Rec}^{T+\delta}$.
By iterating the above argument we see that $\Lambda_{\Src, \Rec}^\infty$
is determined.
\end{proof}

\section{Local reconstruction of the manifold}
\label{sec_local_construction}

In this section we describe a method to reconstruct
$(M_\Rec, g)$ and $(M_B, g)$ from $\Lambda_{\Src, \Rec}$
and $\Lambda_{\Src, B}$, respectively. 
We will first consider the local reconstruction method 
under the additional assumption that the functions
$\sigma_\Rec$ and $\sigma^B$ are known, 
and then show how these functions can be reconstructed from 
the data $\Lambda_{\Src, \Rec}$ and $\Lambda_{\Src, B}$, respectively. 
Let us recall that $\sigma_\Rec(y)$ indicates the distance when the normal geodesic
starting from $y \in \Rec$ hits to the boundary or to a point on the cut locus,
see (\ref{Matti's new formula 1}).
The main difficulty when reconstructing $\sigma_\Rec(y)$
is that the normal geodesic 
may intersect $\p M$ tangentially and this is hard to detect from the data,
see Figure 1. 
To deal with this difficulty we present a method that is based on ``perturbing'' the boundary $\p M$. 


We will next give a series of lemmas that is common for the 
reconstruction method of $\sigma_\Rec$ and $\sigma^B$
and that of $(M_\Rec, g)$ and $(M_B, g)$.
We start by introducing the modified distance function $d_h$.
%
%
%
%
%
%
Let $\Gamma \subset M$ and $h : \Gamma \to \R$.
We define 
\begin{align*}
d_h(x,y) &:= d(x, y) - h(y), \quad x \in M,\ y \in \Gamma,
\\
d_h(x,\Gamma) &:= \inf_{y \in \Gamma} d_h(x, y), \quad x \in M,
\end{align*}
where $d$ is the Riemannian distance function of $(M, g)$.
Moreover, we define the modified domain of influence
\begin{align*}
M(\Gamma, h) & := \{x \in M;\ d_h(x, \Gamma) \le 0\},
\end{align*}
and denote for $T > 0$
\def\B{\mathcal B}
\begin{align*}
\B(\Gamma, h; T) := \{(t, y) \in (0, T) \times \Gamma;\ T - h(y) < t \}. 
\end{align*}

To simplify the notation, we define $M(\Gamma, r)$ 
for a constant $r \in (0, \infty)$ by $M(\Gamma, h)$ 
where $h(y) = r$, $y \in \Gamma$. 
Notice that if $h \in C(\bar \Gamma)$ then
\begin{align*}
M(\Gamma, h) = \{x \in M;\ \text{there is $y \in \bar \Gamma$ such that $d(x, y) \le h(y)$}\},
\end{align*}
and our definition coincides with the definition of the domain of influence in \cite{Oksanen2011}.
In particular, for $\Gamma = \{y\}$ the set
$M(\Gamma, h)$ is the closed geodesic ball with radius $h(y)$.
In this case, we denote also $M(y, h) := M(\Gamma, h)$.

We will show first that $\Lambda_{\Src, \Rec}$ and $\Lambda_{\Src, B}$ determine
certain relations between domains of influences, and 
then that these relations determine $(M_\Rec, g)$
and $(M_B, g)$. The latter step is purely geometric. 

\subsection{From weakly convergent sequences of waves to relations between domains of influences}


Tataru's unique continuation result \cite{Tataru1995} implies that 
the wave equation (\ref{eq_wave}) is approximately controllable, that is, we have the following lemma. 
\begin{lemma}
\label{lem_uniq_cont}
Let $T > 0$ and suppose that
$\Gamma$ is open either in $\p M$ or in $M^\inter$
and that $h \in C(\bar \Gamma)$ satisfies $h \le T$ pointwise.
In the case when $\Gamma \subset M^\inter$ suppose, moreover, that $h > 0$ pointwise.
Then 
\begin{align}
\label{the_dense_set}
\{ u^f(T);\ f \in C_0^\infty(\B(\Gamma, h; T)) \}
\end{align}
is dense in 
\begin{align*}
L^2(M(\Gamma, h)) := \{v \in L^2(M);\ \supp(v) \subset M(\Gamma, h) \}.
\end{align*}
\end{lemma}
Notice that the finite speed of propagation implies that
$u^f(T)$ is supported in $M(\Gamma, h)$ if $f$ is supported in $\B(\Gamma, h; T)$.
In the literature Lemma \ref{lem_uniq_cont} is usually proved only 
in the case of a constant function $h$, see e.g. \cite[Th. 3.10]{Katchalov2001}.
However, the case $h \in C(\bar \Gamma)$ can be reduced to this case by 
approximating $h$ with piecewise constant functions, 
see \cite[Lemmas 4.2 and 4.3]{Oksanen2011}.

\begin{lemma}
\label{lem_domi_test}
Let $T > 0$, $J \in \N$, $j = 1, \dots, J$ 
and let $\Gamma_j$ be open either in $\p M$ or in $M^\inter$.
Let $h_j \in C(\bar \Gamma_j)$ satisfy $h_j \le T$ pointwise
and, in the case $\Gamma_j \subset M^\inter$, also $h_j > 0$.
We define $\Gamma := \bigcup_{j=1}^J \Gamma_j$,
\begin{align}
\label{max_h}
h(y) := \max\{ h_j(y); \text{$j$ satisfies $\bar \Gamma_j \ni y$} \}
\end{align}
and denote $\mathcal U_1(f) := u^f(T)$.
Let $\Gamma_0$ be open either in $\p M$ or in $M^\inter$ and let $s_0 \in (0, T]$.
Let $\mu \in C^\infty(\Gamma \cup \Gamma_0)$ be strictly positive. 
Then the following properties are equivalent:
\begin{itemize}
\item[(i)] $M(\Gamma_0, s_0) \subset \bigcup_{j=1}^J  M(\Gamma_j, h_j)$.
\item[(ii)] For all $f_0 \in C_0^\infty(\B(\Gamma_0, s_0; T))$
there is $(f_j)_{j=1}^\infty \subset C_0^\infty(\B(\Gamma, h; T))$
such that $(\mathcal U_1(\mu(f_0 - f_j)))_{j=1}^\infty$ 
tends to zero weakly in $L^2(M)$.
\end{itemize}
\end{lemma}
\begin{proof}
Notice that $f \mapsto \mu f$ is a bijection on 
$C_0^\infty(\B(\Gamma_0, s_0; T))$ and also on $C_0^\infty(\B(\Gamma, h; T))$.
Thus we lose no generality by assuming that $\mu = 1$ identically. 
We have
\begin{align*}
M(\Gamma, h) 
&= 
\{x \in M;\ \text{there is $y \in \bar \Gamma$ s.t. $d(x, y) \le h(y)$}\}
\\&=
\bigcup_{j=1}^J  M(\Gamma_1, h_j).
\end{align*}
The implication from (i) to (ii) follows from Lemma \ref{lem_uniq_cont}.
We will now show that (ii) implies (i).
We denote 
\begin{align*}
&M_0 := M(\Gamma_0, s_0),
\quad 
M_1 := M(\Gamma, h),
\\
&S_0 := \B(\Gamma_0, s_0; T), 
\quad 
S_1 := \B(\Gamma, h; T).
\end{align*}
Let us assume that (i) does not hold
and let $x \in M_0 \setminus M_1$. 
As $M_1$ is closed, there is a neighborhood $U$ of $x$
such that $U \cap M_1 = \emptyset$.
We will show next that $U \cap M_0^{\inter}$ is nonempty. 

If $x \in \bar \Gamma_0$, then points close to $x$ are in $M_0$ since $s_0 > 0$.
Let us now assume that $x \notin \bar \Gamma_0$.
Then there is a path $\gamma$ from $x$ to a point $y_0 \in \bar \Gamma_0$
such that its length satisfies $0 < l(\gamma) \le s_0$.
We may assume that $\gamma$ is a shortest path between $x$ and $y_0$
and that it has unit speed \cite{Alexander1981}.
Then $\gamma(t) \in U \cap M_0^{\inter}$ for small $t > 0$.

We have shown that $U \cap M_0^{\inter}$ is nonempty.
Hence there is a nonempty open $V \subset M_0$ such that $V \cap M_1 = \emptyset$.
By Lemma \ref{lem_uniq_cont} there is a smooth function 
$f_0$ supported in $S_0$ such that $\int_V u^{f_0}(T) dx \ne 0$.
However, by finite speed of propagation $u^{f}(T)|_V = 0$ for any 
$f$ supported in $S_1$. Thus 
\begin{equation*}
(u^{f_0}(T) - u^{f}(T), 1_V) = (u^{f_0}(T), 1_V) \ne 0,
\end{equation*}
for all $f$ supported in $S_1$ and (ii) does not hold.
\end{proof}

\begin{proposition}
\label{prop_from_weak_conv_to_rels}
Let $\Src, \Rec \subset \p M$ and $B \subset M^\inter$
be open, non-empty sets with smooth boundaries
and suppose that the Hassell-Tao condition (\ref{def_spectral_cond}) holds
with the set $\Src$.
Then $\Lambda_{\Src, \Rec}$ together with the smooth structure of 
$\bar \Src \cup \bar \Rec$ determines the relation  
\begin{align}
\label{boundary_distance_relation_Rec}
\{ (y_0, y_1, r, s, t) \in \Rec^2 \times (0, \infty)^3 ;\ 
M(y_0, r) \subset M(\Rec, s) \cup M(y_1, t) \}.
\end{align}
Moreover, $\Lambda_{\Src, B}$ together with the smooth structure of 
$\bar \Src \cup \bar B$ determines the relation
\begin{align}
\label{boundary_distance_relation_B}
\{ (y_0, y_1, r, s, t) \in \p B^2 \times (0, \infty)^3 ;\ 
M(y_0, r) \subset M(B, s) \cup M(y_1, t) \}.
\end{align}
\end{proposition} 
\begin{proof}
Notice that $M(\Gamma, r) = M = M(\Gamma, T)$
for $r \ge T$ and non-empty $\Gamma \subset M$
since $T \ge \operatorname{diam}(M)$.
The claim follows from Lemmas \ref{lem_blago}, \ref{lem_weak_conv} and \ref{lem_domi_test},
Proposition \ref{prop_weak_boundedness} and the following observation.
Let $y_0, y_1 \in M$, $s_0, s_1 > 0$ and
$J \in \N$. Let $\Gamma_j$ be open either in $\p M$ or in $M^\inter$
and $h_j \in C(\bar \Gamma_j)$ for $j = 1, \dots, J$.
Then the following properties are equivalent:
\begin{itemize}
\item[(i)] $M(y_0, s_0) \subset \bigcup_{j=1}^J M(\Gamma_j, h_j)$.
\item[(ii)] For all $\epsilon > 0$ there is a neighborhood $\Gamma_0$ of $y_0$ 
such that 
\begin{align*}
M(\Gamma_0, s_0) \subset \bigcup_{j=1}^J M(\Gamma_j, h_j + \epsilon).
\end{align*}
\end{itemize}
If $y_0 \in \p M$ then we may take $\Gamma_0 \subset \p M$ in (ii).
Moreover, the following properties are equivalent:
\begin{itemize}
\item[(i')] $M(y_0, s_0) \subset M(y_1, s_1) \cup \bigcup_{j=2}^J M(\Gamma_j, h_j)$.
\item[(ii')] For all neighborhoods $\Gamma_1$ of $y_1$ we have
\begin{align*}
M(y_0, s_0) \subset M(\Gamma_1, s_1) \cup \bigcup_{j=2}^J M(\Gamma_j, h_j).
\end{align*}
\end{itemize}
If $y_1 \in \p M$ then we may take $\Gamma_1 \subset \p M$ in (ii').
\end{proof}

\subsection{From relations between domains of influences to distance functions}
\label{sec_boundary_distance_functions}

In this section we prove that the distances 
\begin{align}
\label{evaluation_dist}
d(\gamma(s; y, \nu), z), 
\quad (s, y) \in \NN_\Rec,\ z \in \Rec,
\end{align}
are determined by $\sigma_\Rec$ and the relation (\ref{boundary_distance_relation_Rec}).
Moreover, we prove an analogous result for the relation (\ref{boundary_distance_relation_B}).
To formulate the result, let us recall 
that we have defined for open $\Gamma \subset \p M$ and $y \in \Gamma$,
\begin{align*}
\sigma_\Gamma(y) := \sup \{ s \in (0, \tau_M(y, \nu)];\ d(\gamma(s; y, \nu), \Gamma) = s\}.
\end{align*}
For open $\Gamma \subset M^\inter$ with smooth boundary and $y \in \p \Gamma$ 
we define $\sigma_\Gamma(y)$ by the same formula where $\nu$ 
is now the interior unit normal vector of $M \setminus \Gamma$.
We will show the following lemma.

\begin{lemma}
\label{lem_distance}
Suppose that one of the following is satisfied.
\begin{itemize}
\item[(a)] $\Gamma \subset \p M$ is open and $y_0 \in \Gamma$.
\item[(b)] $\Gamma \subset M^\inter$ is open with smooth boundary and $y_0 \in \p \Gamma$.
\end{itemize}
Let $y_1 \in M$, $t > 0$ and let
$0 < s \le \sigma_{\Gamma}(y_0)$.
Then the following properties are equivalent:
\begin{itemize}
\item[(i)] $d(\gamma(s; y_0, \nu), y_1) \le t$.
\item[(ii)] For all $\epsilon > 0$ there is $\delta > 0$ such that 
\begin{equation*}
M(y_0, s) \subset M(\Gamma, s - \delta) \cup M(y_1, t + \epsilon).
\end{equation*}
\end{itemize}
\end{lemma}

In Section \ref{sec_cut_locus}
we will reconstruct the function $\sigma_\Rec$.
To this end we will consider here also the modified distance functions $d_h$.
Let $\Gamma \subset \p M$ be open and let $h \in C^1(\bar \Gamma)$
satisfy
\begin{align}
\label{grad_h_small}
|\grad_{\p M} h(y)|_g < 1
\quad \text{for all $y \in \bar \Gamma$}.
\end{align}
We define the modified normal vector,
\begin{align*}
V(h) := \ll( 1 - |\grad_{\p M} h|_g^2 \rr)^{1/2} \nu - \grad_{\p M} h,
\end{align*}
and the modified distance to a cut point
\begin{align*}
\sigma_\Gamma(y; h) := \sup \{ s \in (0, \tau_M(y, V(h))];\ d_h(\gamma(s; y, V(h)), \Gamma) = s\}.
\end{align*}
Notice that $\sigma_\Gamma(y) = \sigma_\Gamma(y; 0)$.

To unify the notation we define $V(h)$ and $\sigma_\Gamma(y; h)$
also for open $\Gamma \subset M^\inter$ with smooth boundary, $y \in \p \Gamma$
and $h = 0$ by the same formulas. That is $V(h) = \nu$
and $\sigma_\Gamma(y; h) = \sigma_\Gamma(y)$
where $\nu$ is now the interior unit normal vector of $M \setminus \Gamma$.

\begin{lemma}
\label{lem_orthogonality}
Let $\Gamma \subset \p M$ be open, $h \in C^1(\bar \Gamma)$, $x \in M^\inter$ 
and suppose that $y_0 \in \Gamma$ satisfies
\begin{align*}
d_h(x, y_0) = d_h(x,\Gamma).
\end{align*}
Moreover, let $\gamma$ be a unit speed shortest path from $y_0$ to $x$. 
If $h$ satisfies (\ref{grad_h_small})
then there is $\rho > 0$ such that $\gamma((0, \rho)) \subset M^\inter$
and $\gamma|[0, \rho]$ is the geodesic with initial velocity $V(h)$.
\end{lemma}
By \cite{Alexander1981} every shortest path is $C^1$, whence we lose no generality with the assumption that $\gamma$ has unit speed. 
\begin{proof}
Let us denote $t := d(x, y_0)$.
We prove the existence of $\rho$ by a contradiction, so suppose that there is a strictly 
decreasing sequence $(s_j)_{j=1}^\infty$ in $(0, t)$ converging to zero
such that $\gamma(s_j) \in \p M$.
Let us consider boundary normal coordinates $(r, z) \in [0, \infty) \times \p M$ 
in a neighborhood of $y_0$.
In these coordinates the metric tensor has the form
\begin{align}
\label{semi_geodesic_coordinates}
g(r, z) 
= dr^2 + g^0(r, z) dz^2 = dr^2 + \sum_{j,k = 1}^{n-1} g^0_{jk}(r, z) dz^j dz^k.
\end{align}
We denote the boundary normal coordinates of $\gamma(s)$ by $(r(s), z(s))$.

Notice that for all $\epsilon > 0$ there is $\delta > 0$ such that for all $s \in [0, \delta]$
\begin{align*}
|\dot z(s)|_{g^0(0, z(s))} \le |\dot z(s)|_{g^0(\gamma(s))} + \epsilon
\le |\dot \gamma(s)|_{g(\gamma(s))} + \epsilon
= 1 + \epsilon.
\end{align*}
Indeed, the first inequality follows from $r(0) = 0$ and smoothness of $g^0$,
and the second one from (\ref{semi_geodesic_coordinates}).
We denote $y_j := \gamma(s_j)$ and $h' := \grad_{\p M} h$. Then for small $\epsilon > 0$, large $j$ and $k > j$,
\begin{align*}
|h(y_j) - h(y_k)| 
&=
|\int_{s_k}^{s_j} \p_s h(z(s)) ds|
\le \norm{h'}_{C(\bar \Gamma)} \int_{s_k}^{s_j} |\dot z(s)|_{g^0(0, z(s))} ds
\\&
= \norm{h'}_{C(\bar \Gamma)} (1 + \epsilon) (s_j - s_k)
< s_j - s_k.
\end{align*}
By taking the limit $k \to \infty$ we have $|h(y_j) - h(y_0)| < s_j$. 
Hence for large $j$
\begin{align*}
d_h(x, y_j) &\le l(\gamma|[s_j, t]) - h(y_j) < t - s_j -h(y_0) + s_j  
\\&= 
d_h(x, y_0) = d_h(x, \Gamma),
\end{align*}
which is a contradiction as $y_j \in \Gamma$ for large $j$.
We have shown that there is $\rho > 0$ such that $\gamma((0, \rho)) \subset M^\inter$.
In particular $\gamma|[0, \rho]$ coincides with a geodesic.
The fact $\dot \gamma(0) = V(h)$ follows from a variation argument, 
see e.g. \cite[p. 99]{Lee1997} for a similar proof.
%
\end{proof}

\begin{lemma}
\label{lem_distance_test}
Let us suppose one of the following 
\begin{itemize}
\item[(a)] $\Gamma \subset \p M$ is open, $y \in \Gamma$ and $h \in C(\bar \Gamma)$ satisfies $h(y) = 0$.
\item[(b)] $\Gamma \subset M^\inter$ is open and has smooth boundary, $y \in \p \Gamma$ and $h = 0$ identically.
\end{itemize}
Let $t > 0$.
Then the following properties are equivalent
\begin{itemize}
\item[(i)] There is $x \in M$ such that $d(x, y) = d_h(x, \Gamma) = t$.
\item[(ii)] For all $s < t$, $M(y, t) \not\subset M(\Gamma, s + h)$.
\end{itemize}
Moreover, if $h \in C^1(\bar \Gamma)$ satisfies (\ref{grad_h_small}) and
$t \le \tau_M(y, V(h))$,
then $x$ in (i) is unique and $x = \gamma(t; y, V(h))$.
\end{lemma}
\begin{proof}
It is clear that (i) implies (ii).
Let us show (i) assuming (ii).
We choose a sequence $(x_j, s_j)_{j=1}^\infty$ in $M(y, t) \times (0,t)$ such that 
$x_j \notin M(\Gamma, s_j + h)$ and $s_j \to t$.
By considering a subsequence we may assume that $x_j \to x$.
As $M(y, t)$ is closed we have $x \in M(y, t)$, whence 
\begin{align*}
d_h(x, y) = d(x, y) \le t.
\end{align*}
Moreover, $x_j \notin M(\Gamma, s_j + h)$ implies
\begin{align*}
d_h(x, y) \ge d_h(x, \Gamma) = \lim_{j \to \infty} d_h(x_j, \Gamma) \ge \lim_{j \to \infty} s_j = t.
\end{align*}
We have shown (i).

Let us proceed to show the uniqueness. Let $\gamma$ be a unit speed shortest path from $y$ to $x$.
Then $\gamma$ coincides with $\gamma(\cdot; y, V(h))$
as long as it does not intersect $\p M$.
Indeed, this follows from Lemma \ref{lem_orthogonality} in the case (a) and 
from an analogous variation argument in the case (b). 
Finally, $d(x, y) = t \le \tau_M(y, V(h))$ implies that 
$x = \gamma(t; y, V(h))$. 
\end{proof}

\begin{lemma}
\label{lem_sigma_test}
Let us suppose one of the following 
\begin{itemize}
\item[(a)] $\Gamma \subset \p M$ is open, $y \in \Gamma$ and $h \in C^1(\bar \Gamma)$ satisfies $h(y) = 0$
and (\ref{grad_h_small}).
\item[(b)] $\Gamma \subset M^\inter$ is open and has smooth boundary, $y \in \p \Gamma$ and $h = 0$ identically.
\end{itemize}
Let $r > 0$.
Then (i) implies (ii) where 
\begin{itemize}
\item[(i)] $r \le \sigma_\Gamma(y; h)$
\item[(ii)] For all $t \in (0, r]$ and $s \in (0, t)$, $M(y, t) \not\subset M(\Gamma, s + h)$.
\end{itemize}
Moreover, if $\gamma(\cdot; y, V(h))$ is transversal to $\p M$ then (ii) implies (i).
\end{lemma}
\begin{proof}
If (i) holds then
$x = \gamma(t; y, V(h))$ satisfies (i) of Lemma \ref{lem_distance_test} for all $t \in (0, r]$,
whence (ii) holds.

Let us now suppose that $\gamma(\cdot; y, V(h))$ is transversal to $\p M$.
We will first prove that $r > \tau_M(y, V(h))$ together with (ii) yield a contradiction.
Let us denote $\tau := \tau_M(y, V(h))$ and
let $t \in (\tau, r)$. By Lemma \ref{lem_distance_test} there is 
$x \in M$ satisfying 
$d(x, y) = d_h(x, \Gamma) = t$.
Thus any shortest path $\gamma$ from $y$ to $x$ coincides 
with $\gamma(\cdot; y, V(h))$ on the interval $[0, \tau]$.
By transversality $\dot \gamma(\tau) \notin T \p M$,
whence $\gamma$ is not $C^1$ at $\tau \in (0, t)$.
This is a contradiction, since $\gamma : [0, t] \to M$ is a shortest path,
whence it is $C^1$, see \cite{Alexander1981}.
We have shown that (ii) implies $r \le \tau_M(y, V(h))$.

Now we see that (ii) implies also (i) 
by applying Lemma \ref{lem_distance_test} for all $t \in (0, r]$.
\end{proof}

\begin{proof}[Proof of Lemma \ref{lem_distance}]
We denote $x_0 := \gamma(s; y_0, \nu)$.
By Lemmas \ref{lem_sigma_test} and \ref{lem_distance_test} 
the point $x_0$ is the only point $x \in M$ satisfying
$d(x, y_0) = d(x, \Gamma) = s$. In particular,
$x_0 \notin M(\Gamma, s - \delta)$ for $\delta > 0$.

If (ii) holds, then $x_0 \in M(y_1, t + \epsilon)$ for all $\epsilon > 0$.
Hence $d(x_0, y_1) \le t + \epsilon$ for all $\epsilon > 0$, and we have (i).

Let us now assume (i) and let $\epsilon > 0$ and $x \in M(y_0, s)$.
If there does not exist $\delta > 0$ such that $x \in M(\Gamma, s - \delta)$, 
then $d(x, \Gamma) > s - \delta$ for all $\delta > 0$.
Thus 
\begin{align*}
s \ge d(x, y_0) \ge d(x, \Gamma) \ge s,
\end{align*}
and $x = x_0 \in M(y_1, t + \epsilon)$.
We have shown (ii).
\end{proof}

\subsection{From distance funtions to local reconstructions of the manifold}
\label{sec_reconstruction_g}

By Lemma \ref{lem_distance} the distances,
\begin{align}
\label{evaluation_dist}
d(\gamma(s; y, \nu), z), 
\quad (s, y) \in \NN_\Rec,\ z \in \Rec,
\end{align}
are determined by $\sigma_\Rec$ and the relation (\ref{boundary_distance_relation_Rec}).
The considerations in \cite[Section 4.4.6]{Katchalov2001} imply that the distances (\ref{evaluation_dist})
determine $(M_\Rec, g)$ in the boundary normal coordinates (\ref{boundary_normal_coords_s_y}).
The reconstruction of $(M_B, g)$ from the relation (\ref{boundary_distance_relation_B}) 
is rather similar. However, 
we will describe it here for the sake of completeness. 

Let $x \in M$ and $\rho > 0$ and denote 
$B := M(x, \rho)$.
Let us suppose that $\rho$ is small enough so that $B$ is contained in a normal neighborhood of $x$ in $M^\inter$
and consider the interior data $\Lambda_{\Src, B}$.

We let $r \in (0, \rho)$ and denote
\begin{align*}
Y_r(\xi) := \gamma(r; x, \xi), \quad Y_r : S_{x} M \to \p M(x, r).
\end{align*}
Then $Y_r$ is a diffeomorphism and
$\gamma(s; Y_r(\xi), \nu) = \gamma(s + r; x, \xi)$.
Notice also that 
\begin{align}
\label{eq_sigma_Bs}
r + \min_{y \in \p M(x, r)} \sigma_{M(x, r)}(y) = \sigma^B.
\end{align}
Lemma \ref{lem_distance} implies that $Y_r$, $\sigma^B$ and the relation 
(\ref{boundary_distance_relation_B}) for $B = M(x,r)$
determine the distances,
\begin{align*}
d(\gamma(s + r; x, \xi_0), \gamma(r; x, \xi_1)), 
\quad \xi_0, \xi_1 \in S_{x} M,\ 0 < s < \sigma^B - r.
\end{align*}

Let us denote by $dist$ the distance function $d$
in the geodesic normal coordinates (\ref{normal_coords_s_xi}).
We have shown that $(B, g)$, $\sigma^B$ and $\Lambda_{\Src, B}$
determine the distances,
\begin{align}
\label{evaluation_dist_B}
dist((s, \eta), (r, \xi)), 
\quad (s, \eta) \in \NN_B,\ (r, \xi) \in (0, \rho) \times S_x M.
\end{align}

\begin{lemma}
\label{lem_diff_of_dist}
Let $(s_0, \eta_0) \in \NN_B$ and let us consider the differentiated distance function 
\begin{align*}
\Phi(r, \xi) := d_{(s, \eta)} dist((s, \eta), (r, \xi))|_{(s, \eta) = (s_0, \eta_0)},
\quad (r, \xi) \in (0, \rho) \times S_x M.
\end{align*}
Then there is $r_0 > 0$ such that 
the image of $(0, r) \times S_x M$ under $\Phi$ is open in $S_{(s_0, \eta_0)}^* \NN_B$
for all $0 < r < r_0$.
\end{lemma}
\begin{proof}
For $y \in M$, $\eta \in S_y M$ and $t > 0$
we denote $\exp_y(t\eta) := \gamma(t; y, \eta)$.
As $(s_0, \eta_0)$ is in a normal coordinate neighborhood of $x = 0$
the same is true for $(r, \xi)$ with $r$ small enough. Thus 
\begin{align*}
\Phi^\sharp(r, \xi) := \grad_{(s, \eta)} dist((s, \eta), (r, \xi))|_{(s, \eta) = (s_0, \eta_0)} 
= -P \exp_{(s_0, \eta_0)}^{-1}(r, \xi),
\end{align*}
where $P$ is the projection,
\begin{align*}
P v := \frac{v}{|v|_g}, \quad P : T_{(s_0, \eta_0)} \NN_B \to S_{(s_0, \eta_0)} \NN_B.
\end{align*}
%
%
%
As $\exp_{(s_0, \eta_0)}^{-1}$ is a local diffeomorphism around $0$ and $P$ is an open map,
we see that $\Phi^\sharp((0, r) \times S_x M)$ is open in $S_{(s_0, \eta_0)} \NN_B$ for small enough $r$.
The claim follows by using the isomorphism 
$S_{(s_0, \eta_0)} \NN_B \to S_{(s_0, \eta_0)}^* \NN_B$ induced by the metric $g$.
\end{proof}

Lemma \ref{lem_diff_of_dist} implies that 
the second order homogeneous polynomial 
$g(s_0, \eta_0)$ is determined on the space $T_{(s_0, \eta_0)}^* \NN_B = \R^n$ 
by the distances (\ref{evaluation_dist_B}). Thus $g$ is determined also on $T \NN_B$.
To summarize, when we are given $\Lambda_{\Src,B}$, $\sigma^B$ and $(B,g)$
we can determine 
$(M_B, g)$ in the geodesic normal coordinates (\ref{normal_coords_s_xi}).

\subsection{Reconstruction of distances to cut points using modified distance functions}
\label{sec_cut_locus}

In this section we show that for open $\Gamma \subset \p M$
the distance to a cut point $\sigma_\Gamma$ is determined by the relation
\begin{align*}
\{ (y, t, h) \in \Gamma \times (0, \infty) \times C^1(\bar \Gamma);\ M(y, t) \subset M(\Gamma, h) \}.
\end{align*}
Moreover, we show that for a small ball $B \subset M^\inter$ 
the cut time $\sigma^B$ is determined by the relation
\begin{align*}
\{ (y, t, s) \in \Gamma \times (0, \infty)^2 ;\ M(y, t) \subset M(B, s) \}.
\end{align*}

\begin{lemma}
\label{lem_Klingenberg_cont}
Let $\Gamma \subset \p M$ be open and let $y_0 \in \Gamma$. Then the map 
\begin{align*}
\sigma_\Gamma : \Gamma \times C^1(\bar \Gamma) \to \R
\end{align*}
is lower semicontinuous at $(y_0; 0)$.
Moreover, if $\sigma_\Gamma(y_0; 0) < \tau_M(y_0, \nu)$ then 
$\sigma_\Gamma$ is continuous at $(y_0; 0)$.
\end{lemma}
\begin{proof}
We prove the semicontinuity by a contradiction, so suppose that there is a sequence 
$((y_j, h_j))_{j=1}^\infty$ converging to $(y_0, 0)$
such that $\liminf_{j \to \infty} \sigma_\Gamma(y_j; h_j) < \sigma_\Gamma(y_0; 0)$.
We denote $h_0 = 0$ and
\begin{align*}
\sigma_j := \sigma_\Gamma(y_j; h_j), 
\quad \tau_j := \tau_M(y_j, V(h_j)),
\quad \text{for $j \ge 0$}.
\end{align*}
As $[0, \sigma_0]$ is compact, we may consider a subsequence and assume 
that $\sigma_j \to \sigma_\infty$. 
We let $T \in (\sigma_\infty, \sigma_0)$ and define
\begin{align*}
x_j := \gamma(T; y_j, V(h_j)),
\quad
t_j := d_{h_j}(x_j, \Gamma).
\end{align*}
Notice that $x_j$ is well defined for large $j$. Indeed,
the exit time function $\tau_M$ is lower semicontinuous, see e.g. \cite{Helin2010}, whence
\begin{align*}
\liminf_{j \to \infty} \tau_j \ge \tau_0 \ge \sigma_0 > T.
\end{align*}

We denote $x_0 := \gamma(T; y_0, \nu)$.
Then continuity properties of the modified distances imply
$t_j \to d(x_0, \Gamma)$.
Moreover, $T > \sigma_\infty$ implies $t_j < T$ for large $j$,
and $T < \sigma_0$ implies $d(x_0, \Gamma) = T$.
To summarize, we may consider a subsequence and assume that
\begin{align}
\label{ineq_semicont_T}
t_j < \lim_{j \to \infty} t_j = d(x_0, \Gamma) = T < \tau_j.
\end{align}

There is $z_j \in \bar \Gamma$ such that
$d_{h_j}(x_j, z_j) = t_j$.
By considering a subsequence we may assume that $z_j \to z_\infty \in \bar \Gamma$.
We see that $z_\infty$ is a closest point to $x_0$ in $\bar \Gamma$ since
\begin{align*}
d(x_0, z_\infty) 
= \lim_{j \to \infty} d_{h_j} (x_j, z_j) 
= \lim_{j \to \infty} t_j
= T.
\end{align*}
However, $T < \sigma_0$ implies that $z_\infty = y_0$,
see e.g. \cite[pp. 144, 115]{Chavel2006}. 
%
%
%
By considering a subsequence we may assume that $z_j \in \Gamma$ since $z_j \to y_0 \in \Gamma$.
Lemma \ref{lem_orthogonality} and the inequality (\ref{ineq_semicont_T})
imply that
\begin{align*}
x_j = \gamma(t_j; z_j, V(h_j)).
\end{align*}

As $T < \sigma_0$, the map $(r, y) \mapsto \gamma(r; y, \nu)$ 
is a local diffeomorphism at 
$(T, y_0) \in (0, \infty) \times \Gamma$, 
see e.g. \cite[p. 144, Th. III.2.2]{Chavel2006}.
Moreover, the map
\begin{align*}
&\alpha : C^1(\bar \Gamma) \times (0, \infty) \times \Gamma
\to C^1(\bar \Gamma) \times M, 
\\&
\alpha(h, r, y) := (h, \gamma(r; y, V(h)))
\end{align*}
is a local diffeomorphism at $(0, T, y_0)$ since its
derivative is of the form 
\begin{align*}
\left( \begin{array}{cc}
Id & 0
\\
A & d_{(r, y)}  \gamma(r; y, \nu)|_{r=T, y=y_0}
\end{array} \right),
\end{align*}
where $A : C^1(\bar \Gamma) \to T_{x_0} M$ is a continuous linear operator. 
In particular, there is a local inverse $\beta$ such that
in a neighborhood of $(0, T, y_0)$ we have
$\beta(h, \gamma(r; y, V(h))) = (r, y)$.
Hence for large $j$
\begin{align*}
(t_j, z_j) &= \beta(h_j, \gamma(t_j; z_j, V(h_j)))
= \beta(h_j, x_j) = \beta(h_j, \gamma(T; y_j, V(h_j)))
\\&= (T, y_j),
\end{align*}
which is a contradiction with (\ref{ineq_semicont_T}).
We have shown that $\sigma_\Gamma$ is lower semicontinuous at $(y_0; 0)$.

Let us now suppose that $\sigma_0 < \tau_0$ 
and show upper semicontinuity by a contradiction.
To that end, we suppose that $\sigma_\infty > \sigma_0$. 
%
We let $\epsilon > 0$ satisfy
$\sigma_0 + \epsilon < \min(\sigma_\infty, \tau_0)$ 
and denote $x_j(\epsilon) = \gamma(\sigma_0 + \epsilon; y_j, V(h_j))$.
Then for large $j$
\begin{align*}
\sigma_0 + \epsilon = d_{h_j}(x_j(\epsilon), \Gamma). 
\end{align*}
In particular, for small $\epsilon > 0$
\begin{align*}
\sigma_0 + \epsilon = \lim_{j \to \infty} d_{h_j}(x_j(\epsilon), \Gamma) 
= d(\gamma(\sigma_0 + \epsilon; y_0, \nu), \Gamma),
\end{align*}
which is a contradiction with the definition of $\sigma_0$.
\end{proof}

\begin{lemma}
\label{lem_transversality}
Let $\Gamma \subset \p M$ be open and let $y_0 \in \Gamma$.
Then there is $(y_j, h_j)_{j=1}^\infty \subset \Gamma \times C^1(\bar \Gamma)$
converging to $(y_0, 0)$ such that the geodesic $s \mapsto \gamma(s; y_j, V(h_j))$ is traversal to $\p M$,
$h_j(y_j) = 0$ and
\begin{align*}
\lim_{j \to \infty} \sigma_\Gamma(y_j; h_j) = \sigma_\Gamma(y_0; 0).
\end{align*}
In particular, 
\begin{align*}
\liminf_{(y, h) \to (y_0, 0)} \sigma_\Gamma(y; h) = \sigma_\Gamma(y_0; 0).
\end{align*}
\end{lemma}
\begin{proof}
By \cite[Lem. 12]{Helin2010} there is a sequence $(y_j, \eta_j)_{j=1}^\infty \subset \p_- S M$ converging to $(y_0, \nu)$ 
such that $\gamma(\cdot; y_j, \eta_j)$ is transversal to $\p M$ and 
$\tau_M(y_j, \eta_j)$ converges to $\tau_M(y_0, \nu)$ as $j \to \infty$.
We may choose $(h_j)_{j=1}^\infty \subset C^1(\bar \Gamma)$ converging to zero 
such that 
\begin{align*}
h_j(y_j) = 0 
\quad \text{and} \quad 
\grad_{\p M} h_j(y_j) = \eta_j|_{T^* \p M}.
\end{align*}

In the case when $\sigma_\Gamma(y_0; 0) < \tau_M(y_0, \nu)$,
the claim follows immediately from Lemma \ref{lem_Klingenberg_cont}.
Let us consider the case $\sigma_\Gamma(y_0; 0) = \tau_M(y_0, \nu)$.
Then 
\begin{align*}
\sigma_\Gamma(y_0; 0) 
&= 
\tau_M(y_0, \nu) 
= 
\liminf_{j \to \infty} \tau_M(y_j, V(h_j))
\ge
\liminf_{j \to \infty} \sigma_\Gamma(y_j; h_j) 
\\&\ge 
\sigma_\Gamma(y_0; 0).
\end{align*}
Moreover, by considering a subsequence we may assume that 
$\sigma_\Gamma(y_j; h_j)$ converges to $\liminf_{j \to \infty} \sigma_\Gamma(y_j; h_j)$
as $j \to \infty$.
The second claim follows from the first claim and the lower semicontinuity of $\sigma_\Gamma$ at $(y_0; 0)$.
\end{proof}

For open $\Gamma \subset \p M$, $y \in \Gamma$ and $h \in C(\bar \Gamma)$ we define
\begin{align*}
\tilde \sigma_\Gamma(y; h) := \sup \{ t \in (0, \infty);\ \text{$t$ satisfies (ii) of Lemma \ref{lem_distance_test}}\}.
\end{align*}

\begin{lemma}\label{lem: det A}
Let $\Gamma \subset \p M$ be open and $y \in \Gamma$. Then
\begin{align*}
\liminf_{(y, h) \to (y_0, 0)} \tilde \sigma_\Gamma(y; h) = \sigma_\Gamma(y_0; 0),
\end{align*}
where the $\liminf$ is taken over all $(y, h) \in \Gamma \times C^1(\bar \Gamma)$ such that $h(y) = 0$.
\end{lemma}
\begin{proof}
Lemmas \ref{lem_sigma_test} and \ref{lem_transversality} imply 
\begin{align*}
\liminf_{(y, h) \to (y_0, 0)} \tilde \sigma_\Gamma(y; h) \ge \liminf_{(y, h) \to (y_0, 0)} \sigma_\Gamma(y; h)
= \sigma_\Gamma(y_0; 0).
\end{align*}
Let $(h_j)_{j=1}^\infty$ be as in Lemma \ref{lem_transversality}.
Then by Lemma \ref{lem_sigma_test}
\begin{align*}
\liminf_{(y, h) \to (y_0, 0)} \tilde \sigma_\Gamma(y; h) 
\le \liminf_{j \to \infty} \tilde \sigma_\Gamma(y_j; h_j) 
= \lim_{j \to \infty} \sigma_\Gamma(y_j; h_j)
 = \sigma_\Gamma(y_0; 0).
\end{align*}
\end{proof}

\begin{lemma}
\label{lem_sigma_test_B}
Let $x \in M^\inter$ and let $\rho > 0$ be small enough so that
$B := M(x, \rho)$ is contained in a normal neighborhood of $x$ in $M^\inter$.
For $r > 0$ the following properties are equivalent:
\begin{itemize}
\item[(i)] $r + \rho \le \sigma^B$.
\item[(ii)] For all $t \in (0, r]$, $s \in (0, t)$ and $y \in \p B$, $M(y, t) \not\subset M(B, s)$.
\end{itemize}
\end{lemma}
\begin{proof}
By Lemma \ref{lem_sigma_test} and (\ref{eq_sigma_Bs}) is enough to show that 
$\gamma(\cdot; y, \nu)$ is transversal to $\p M$ if 
$y \in \p B$ satisfies $\tau := \tau_M(y, \nu) = \min_{z \in \p B} \tau_M(z, \nu)$.
Let $\xi \in S_{x} M$ satisfy $\gamma(\rho; x, \xi) = y$.
Then 
\begin{align*}
\gamma(\tau; y, \nu) = \gamma(\tau + \rho; x, \xi)
\end{align*}
is a closest point to $y_0$ on $\p M$.
Thus $\gamma(\cdot; y_0, \nu)$ intersects $\p M$ normally, in particular, it is transversal to $\p M$.
\end{proof}

Summarizing, when we are given $\Lambda_{\Src,\Rec}$, 
we may use Proposition \ref{prop_from_weak_conv_to_rels}
to first determine $\sigma_\Rec$ by Lemma \ref{lem: det A}
and then to reconstruct the subset $(M_\Rec, g)$
by Lemma \ref{lem_distance}
and \cite[Section 4.4.6]{Katchalov2001}.
Analogously, when $B$ is as in Lemma \ref{lem_sigma_test_B}
and we are given $\Lambda_{\Src,B}$ together with $(B, g)$,
we may use Proposition \ref{prop_from_weak_conv_to_rels}
to first determine $\sigma^B$ by Lemma \ref{lem_sigma_test_B}
and then to reconstruct the subset $(M_B, g)$
by Lemma \ref{lem_distance}
and Section \ref{sec_reconstruction_g}.

\section{Global reconstruction of the manifold}
\label{sec_global_reconstruction}

In the previous section we have shown, 
under the assumptions of Proposition \ref{prop_from_weak_conv_to_rels},
that $\Lambda_{\Src,\Rec}$ determines $(M_\Rec, g)$.
In fact, we have described a method to reconstruct $(M_\Rec, g)$
in the boundary normal coordinates. 
We will show next that $\Lambda_{\Src,\Rec}$
determines the interior data $\Lambda_{\Src,B}$
for such a ball $B \subset M_\Rec$ that 
there is $\Gamma \subset \Rec$ and $t_0 > 0$ satisfying 
\begin{align*}
B \subset M(\Gamma, t_0) \subset M_\Rec.
\end{align*}
Let $T > t_0$. 
By the finite speed of propagation for the wave equation (\ref{eq_wave}),
$(M_\Rec, g)$ determines $u^f(T)$ for all $f \in C_0^\infty((T-t_0, T) \times \Gamma)$.
Moreover, as the functions 
\begin{align*}
u^f(T), \quad f \in C_0^\infty((T-t_0, T) \times \Gamma),
\end{align*}
are dense in $L^2(M(\Gamma, t_0))$, we see using Lemma \ref{lem_blago} that 
$\Lambda_{\Src,\Rec}$ determines $u^\psi(T)|_B$ for $\psi \in C_0^\infty((t_0, \infty) \times \Src)$.
By varying $T > t_0$ and noticing that the equation (\ref{eq_wave}) is invariant with respect to translation in time, we see that $u^\psi|_{(0, \infty) \times B}$ is determined for 
$\psi \in C_0^\infty((0, \infty) \times \Src)$. That is, $\Lambda_{\Src,\Rec}$ determines $\Lambda_{\Src,B}$.

Let $x \in M^\inter$ and suppose that $g$ is known in a neighborhood of $x$.
Then we can choose $\rho > 0$ small enough so that $B := M(x, \rho)$
is contained in a known normal neighborhood of $x$ in $M^\inter$.
Let us suppose that $\Lambda_{\Src, B}$ is also known.
We have seen that $\Lambda_{\Src,B}$ determines $(M_B, g)$.
Let $B' \subset M_B$ be a ball. Then there is $t_0 > 0$ such that
\begin{align*}
B' \subset M(B, t_0) \subset M_B.
\end{align*}
An argument analogous with the argument above shows that
$\Lambda_{\Src,B}$ determines $\Lambda_{\Src,B'}$.

Let us now show how local reconstructions $(M_B, g)$ and $(M_{B'}, g)$
can be glued together.
Let $x_1 \in M_B$ and $x_2 \in M_{B'}$.
We have seen that the maps $\Lambda_{\Src,B}$ and $\Lambda_{\Src,B'}$
determine the maps $\Lambda_{\Src,B(x_j, \rho)}$, $j=1,2$,
for small $\rho > 0$.
By Lemma (\ref{lem_blago}) we can compute the inner products
\begin{align*}
&(u^{f_1 - f_2}(T), u^\psi(T))_{L^2(M)} 
\\&\quad\quad= 
(u^{f_1}(T), u^\psi(T))_{L^2(M)} - (u^{f_2}(T), u^\psi(T))_{L^2(M)},
\end{align*}
for $f_j \in C_0^\infty((0, \infty) \times B(x_j, \rho))$,
and $\psi \in C_0^\infty((0, \infty) \times \Src)$.
The set
\begin{align*}
S_\rho := \{s \in (0, \infty);\ M(B(x_1, \rho), s) \subset M(B(x_2, \rho), s)\}
\end{align*}
is determined by Proposition \ref{prop_from_weak_conv_to_rels}. 
Moreover, $x = y$ if and only if $S_\rho = (0, \infty)$ for all $\rho > 0$.
That is, we know how to identify points in $M_B \cap M_{B'}$.
In particular, we can reconstruct the transition functions,
whence we have constructed $(M_B \cup M_{B'}, g)$.

Let us consider the collection
\begin{align*}
\mathcal M := \{B;\ &\text{$B \subset M^\inter$ is a ball,}
\\&\text{$\Lambda_{\Src,\Rec}$ determines $(M_B, g)$ and $\Lambda_{\Src,B}$}\}.
\end{align*}
We have shown that $\mathcal M$ is nonempty and that 
\begin{align*}
\text{if $x \in M_B$ and $B \in \mathcal M$ then $M(x, \rho) \in \mathcal M$ for small $\rho >0$.}
\end{align*}
In particular, the open set
\begin{align*}
U := \bigcup_{B \in \mathcal M} M_B \subset M^\inter
\end{align*}
is nonempty. We will show next that it is closed in $M^\inter$.
Let $x \in M^\inter$ and let $x_j \in U$, $j \in \N$,
converge to $x \in M^\inter$. Then there is a uniformly normal neighborhood $W$ of $x$ and $j \in \N$
such that $x_j \in W$, see e.g. \cite[Lem. 5.12]{Lee1997}.
For small $\rho > 0$, $B' := M(x_j, \rho) \in \mathcal M$ since
$x_j \in M_B$ for some $B \in \mathcal M$.
As $W$ is uniformly normal, we have $W \subset M_{B'}$. 
Hence $x \in U$ and we have shown that $U$ is closed in $M^\inter$.
As $M^\inter$ is assumed to be connected, we have $U = M^\inter$.

We have shown that $(M^\inter, g)$ is determined by $\Lambda_{\Src,\Rec}$.
The smooth Riemannian structure allows us to recover also the closure $(M, g)$, 
see e.g. \cite[p. 2116]{Isozaki2010}.  This proves Theorem \ref{thm_main}.
\smallskip

{\em Acknowledgements.}
The research was partly supported by Finnish Centre of Excellence in Inverse Problems Research,
Academy of Finland project COE 250215.
L.O. was partly supported also by European Research Council advanced grant 400803.


\bibliographystyle{abbrv} 
\bibliography{main}
\end{document}

%% file: dogbone.pdf_tex

\begingroup
  \makeatletter
  \providecommand\color[2][]{%
    \errmessage{(Inkscape) Color is used for the text in Inkscape, but the package 'color.sty' is not loaded}
    \renewcommand\color[2][]{}%
  }
  \providecommand\transparent[1]{%
    \errmessage{(Inkscape) Transparency is used (non-zero) for the text in Inkscape, but the package 'transparent.sty' is not loaded}
    \renewcommand\transparent[1]{}%
  }
  \providecommand\rotatebox[2]{#2}
  \ifx\svgwidth\undefined
    \setlength{\unitlength}{172.89330206pt}
  \else
    \setlength{\unitlength}{\svgwidth}
  \fi
  \global\let\svgwidth\undefined
  \makeatother
  \begin{picture}(1,1.55117699)%
    \put(0,0){\includegraphics[width=\unitlength]{dogbone.pdf}}%
  \end{picture}%
\endgroup

%% file: dogbonexy2.pdf_tex

\begingroup
  \makeatletter
  \providecommand\color[2][]{%
    \errmessage{(Inkscape) Color is used for the text in Inkscape, but the package 'color.sty' is not loaded}
    \renewcommand\color[2][]{}%
  }
  \providecommand\transparent[1]{%
    \errmessage{(Inkscape) Transparency is used (non-zero) for the text in Inkscape, but the package 'transparent.sty' is not loaded}
    \renewcommand\transparent[1]{}%
  }
  \providecommand\rotatebox[2]{#2}
  \ifx\svgwidth\undefined
    \setlength{\unitlength}{108.82963542pt}
  \else
    \setlength{\unitlength}{\svgwidth}
  \fi
  \global\let\svgwidth\undefined
  \makeatother
  \begin{picture}(1,2.32463386)%
    \put(0,0){\includegraphics[width=\unitlength]{dogbonexy2.pdf}}%
    \put(0.57297751,2.12703284){\color[rgb]{0,0,0}\makebox(0,0)[lb]{\smash{$y$}}}%
    \put(0.3600872,0.19804075){\color[rgb]{0,0,0}\makebox(0,0)[lb]{\smash{$x$}}}%
  \end{picture}%
\endgroup